\documentclass[12pt]{amsart}

\input xy
\xyoption{all}

\usepackage{amssymb}

\newtheorem{theorem}{Theorem}[section]
\newtheorem{proposition}[theorem]{Proposition}
\newtheorem{corollary}[theorem]{Corollary}
\newtheorem{definition}[theorem]{Definition}

\numberwithin{equation}{section}

\begin{document}

\baselineskip=15pt

\title[Chern classes of parabolic vector bundles]{A construction
of Chern classes of parabolic vector bundles}

\author[I. Biswas]{Indranil Biswas}

\address{School of Mathematics, Tata Institute of Fundamental
Research, Homi Bhabha Road, Bombay 400005, India}

\email{indranil@math.tifr.res.in}

\author[A. Dhillon]{Ajneet Dhillon}

\address{Department of Mathematics, University of Western Ontario,
London, Ontario N6A 5B7, Canada}

\email{adhill3@uwo.ca}



\date{}

\begin{abstract}
Given a parabolic vector bundle, we construct for it of projectivization
and tautological line bundle. These are analogs of the projectivization
and tautological line bundle for an usual vector bundle. Using these
we give a construction of the parabolic Chern
classes.
\end{abstract}

\maketitle

\section{Introduction}

Parabolic vector bundles on a smooth complex projective curve were introduced
in \cite{Se}. In \cite{MY}, Maruyama and Yokogawa introduced parabolic vector
bundles on higher dimensional complex projective varieties. The notion of Chern
classes of a vector bundle extends to the context of parabolic
vector bundles \cite{Bi-c}, \cite{IS}, \cite{Ta}.

Take a vector bundle $V$ of rank $r$ on a variety $Z$. Let $\psi\, :\, 
{\mathbb P}(V)\,\longrightarrow\, Z$ be the projective bundle parametrizing
hyperplanes in the fibers of the vector bundle $V$.
The tautological line bundle on 
${\mathbb P}(V)$ will be denoted by ${\mathcal O}_{{\mathbb P}(V)}(1)$.
One of the standard ways of constructing Chern classes of
$V$ is to use the identity
$$
\sum_{i=0}^r (-1)^i c_1({\mathcal O}_{{\mathbb P}(V)}(1))^{r-i}
\psi^*c_i(V)\,=\, 0
$$
with $c_0(V)\,=\, 1$.

Our aim here is to give a construction of Chern classes of
parabolic vector bundles along this line (see Theorem \ref{thm1}).

N. Borne showed that parabolic vector bundles can be understood as
vector bundles on root-stacks \cite{Bo1}, \cite{Bo2}. In terms of this
correspondence, the $i$-th parabolic Chern class of a parabolic vector
bundle is the usual $i$-th Chern class of the corresponding vector bundle
on root-stack. It should be mentioned that this elegant correspondence
in \cite{Bo1}, \cite{Bo2} between parabolic vector bundles and vector
bundles on root-stacks is turning out to be very useful
(see, for example, \cite{BD} for application of this correspondence).

\section{Preliminaries}

\subsection{Parabolic vector bundles}

Let $X$ be an irreducible smooth projective variety defined over
$\mathbb C$. Let $D\, \subset\, X$ be an effective reduced divisor
satisfying the condition that each irreducible component
of $D$ is smooth, and the irreducible components of $D$ intersect
transversally; divisors satisfying these conditions are called
simple normal crossing ones.
Let
\begin{equation}\label{decomp.-D}
D\, =\, \sum_{i=1}^\ell D_i
\end{equation}
be the decomposition of $D$ into irreducible components.

Let $E_0$ be an algebraic vector bundle over $X$. For
each $i\,\in\, [1\, , \ell]$, let
\begin{equation}\label{divisor-filt.}
E_0\vert_{D_i}\, =\, F^i_1 \,\supsetneq\, F^i_2 \,\supsetneq\,
\cdots \,\supsetneq\, F^i_{m_i} \,\supsetneq\, F^i_{m_i+1}\, =\, 0
\end{equation}
be a filtration by subbundles of the restriction of $E_0$ to $D_i$.

A \textit{quasiparabolic} structure on $E_0$ over $D$ is a filtration
as above of each $E_0\vert_{D_i}$ such that the system of filtrations
is locally abelian (see \cite[p. 157, Definition 2.4.19]{Bo2} for the 
definition of a locally abelian structure).

For a quasiparabolic structure as above, \textit{parabolic 
weights} are a collection of rational numbers
$$
0\, \leq\, \lambda^i_1\, < \, \lambda^i_2\, < \,
\lambda^i_3 \, < \,\cdots \, < \, \lambda^i_{m_i}
\, <\, 1\, ,
$$
where $i\,\in\, [1\, ,\ell]$. The parabolic weight $\lambda^i_j$
corresponds to the subbundle $F^i_j$ in \eqref{divisor-filt.}. A 
\textit{parabolic structure} on $E_0$ is a quasiparabolic
structure on $E_0$ (defined as above) equipped with parabolic weights.
A vector bundle over $X$ equipped
with a parabolic structure on it is also called a
\textit{parabolic vector bundle}. (See \cite{MY}, \cite{Se}.)
For notational convenience, a parabolic vector bundle defined
as above will be denoted by $E_*$.

The divisor $D$ is called the \textit{parabolic divisor} for
$E_*$. We fix $D$ once and for all. So the parabolic divisor
of all parabolic vector bundles on $X$ will be $D$.

The definitions of direct sum, tensor product and dual of vector
bundles extend naturally to parabolic vector bundles; similarly,
symmetric and exterior powers of parabolic vector bundles are also
constructed (see \cite{MY}, \cite{Bi2}, \cite{Yo}).

\subsection{Ramified principal bundles}

The complement of $D$ in $X$ will be denoted by $X-D$.

Let
$$
\varphi\, :\, E_{\text{GL}(r, {\mathbb C})}\, \longrightarrow\, X
$$
be a ramified principal $\text{GL}(r, {\mathbb C})$--bundle
with ramification over $D$
(see \cite{BBN}, \cite{Bi2}, \cite{Bi3}). We
briefly recall its defining properties.
The total space $E_{\text{GL}(r, {\mathbb C})}$ is a
smooth complex variety equipped with
an algebraic right action of $\text{GL}(r, {\mathbb C})$
\begin{equation}\label{f}
f\, :\,E_{\text{GL}(r, {\mathbb C})}\times \text{GL}(r,
{\mathbb C})\, \longrightarrow\, E_{\text{GL}(r, {\mathbb C})}\, ,
\end{equation}
and the following conditions hold:
\begin{enumerate}
\item{} $\varphi\circ f \, =\, \varphi\circ p_1$, where $p_1$ is
the natural projection of $E_{\text{GL}(r, {\mathbb C})}\times
\text{GL}(r, {\mathbb C})$ to $E_{\text{GL}(r, {\mathbb C})}$,

\item{} for each point $x\, \in\, X$, the action of
$\text{GL}(r, {\mathbb C})$ on the
reduced fiber $\varphi^{-1}(x)_{\text{red}}$ is transitive,

\item{} the restriction of $\varphi$ to $\varphi^{-1}(X
- D)$ makes $\varphi^{-1}(X- D)$ a principal
$\text{GL}(r, {\mathbb C})$--bundle over $X- D$,

\item{} for each irreducible component $D_i\, \subset\, D$,
the reduced inverse image $\varphi^{-1}(D_i)_{\text{red}}$ is a
smooth divisor and
$$
\widehat{D}\, :=\, \sum_{i=1}^\ell \varphi^{-1}(D_i)_{\text{red}}
$$
is a normal crossing divisor on $E_{\text{GL}(r, {\mathbb C})}$, and

\item{} for any point $x$ of $D$, and any point
$z\, \in\, \varphi^{-1}(x)$, the isotropy
group
\begin{equation}\label{e8}
G_z\, \subset\,\text{GL}(r, {\mathbb C}) \, ,
\end{equation}
for the action of $\text{GL}(r, {\mathbb C})$ on
$E_{\text{GL}(r, {\mathbb C})}$, is a finite group, and if
$x$ is a smooth point of $D$, then the natural action of 
$G_z$ on the quotient line $T_zE_{\text{GL}(r,
{\mathbb C})}/T_z\varphi^{-1}(D)_{\text{red}}$ is faithful.
\end{enumerate}

Let
$$
D^{\rm sm}\, \subset\, D
$$
be the smooth locus of the divisor. Take any point $x\, \in\,D^{\rm sm}$,
and let $z\, \in\, \varphi^{-1}(x)$ be any point.
Since $G_z$ acts faithfully on the line
$T_zE_{\text{GL}(r, {\mathbb C})}/T_z\varphi^{-1}(D)_{\text{red}}$,
it follows that $G_z$ is a cyclic group.
Take any $z'\,\in\, E_{\text{GL}(r, {\mathbb C})}$ such that
$\varphi(z')\, =\, \varphi(z)$. There is an element $g\,\in\,
\text{GL}(r, {\mathbb C})$ such that $f(z\, ,g)\,=\, z'$. Therefore,
the subgroup $G_z$ is conjugate to the subgroup $G_{z'}$; more
precisely, we have $g^{-1}G_zg\,=\, G_{z'}$. In particular,
$G_z$ is isomorphic to $G_{z'}$.

There is a natural bijective correspondence between the ramified
principal $\text{GL}(r, {\mathbb C})$--bundles with ramification
over $D$ and the parabolic vector bundles of rank $r$ (see
\cite{BBN}, \cite{Bi2}). We first describe the steps in the construction
of a ramified principal $\text{GL}(r, {\mathbb C})$--bundle from a
parabolic vector bundle of rank $r$:
\begin{itemize}
\item Given a parabolic vector bundle $E_*$ of rank $r$ on $X$, there is
a Galois covering
\begin{equation}\label{e1}
\gamma\, :\, Y\, \longrightarrow\, X\, ,
\end{equation}
where $Y$ is
an irreducible smooth projective variety, and a $\text{Gal}(\gamma)$--linearized
vector bundle $F$ on $Y$ \cite{Bi1}, \cite{Bo1}, \cite{Bo2}. Let 
$F_{\text{GL}(r, {\mathbb C})}$ be the principal $\text{GL}(r, {\mathbb 
C})$--bundle on $Y$ defined by $F$. We recall that $F_{\text{GL}(r,
{\mathbb C})}$ is the space of all linear isomorphisms from ${\mathbb C}^r$
to the fibers of $F$.

\item The linearization action of $\text{Gal}(\gamma)$ on $F$ produces an
action of $\text{Gal}(\gamma)$ on $F_{\text{GL}(r, {\mathbb C})}$. This
action of $\text{Gal}(\gamma)$ on $F_{\text{GL}(r, {\mathbb C})}$ commutes
with the action of $\text{GL}(r, {\mathbb C})$ on $F_{\text{GL}(r, {\mathbb 
C})}$ because it is given by a linearization action on $F$.

\item The quotient $$\text{Gal}(\gamma)\backslash F_{\text{GL}(r, {\mathbb C})}
\, \longrightarrow\, \text{Gal}(\gamma)\backslash Y\,=\, X$$
is a ramified principal $\text{GL}(r, {\mathbb C})$--bundle.
\end{itemize}
It is starightforward to check that
$\text{Gal}(\gamma)\backslash F_{\text{GL}(r, {\mathbb C})}$ is a 
ramified principal $\text{GL}(r, {\mathbb C})$--bundle over $X$.

We will now describe the construction of a parabolic vector bundles of rank
$r$ from a ramified principal $\text{GL}(r, {\mathbb C})$--bundle.

Let
$$
\varphi \, :\, F_{\text{GL}(r, {\mathbb C})}\, \longrightarrow\, X
$$
be a ramified principal $\text{GL}(r, {\mathbb C})$--bundle. Let
$f$ be as in \eqref{f}. Consider the trivial vector bundle
$$
W\, :=\, F_{\text{GL}(r, {\mathbb C})}\times {\mathbb C}^r\,\longrightarrow\,
F_{\text{GL}(r, {\mathbb C})}\, .
$$
The group $\text{GL}(r,{\mathbb C})$ acts on $F_{\text{GL}(r, {\mathbb C})}
\times {\mathbb C}^r$ as follows: the action of any $g\, \in\,
\text{GL}(r,{\mathbb C})$ sends any $(z\, ,v)\, \in\, F_{\text{GL}(r,
{\mathbb C})} \times {\mathbb C}^r$ to $(f(z,g)\, ,g^{-1}(v))$. Note that
this action on $W$ is a lift of the action of $\text{GL}(r,{\mathbb C})$ on
$F_{\text{GL}(r, {\mathbb C})}$ defined by $f$. This action
of $\text{GL}(r,{\mathbb C})$ on $W$ produces an action of
$\text{GL}(r,{\mathbb C})$ on the quasicoherent sheaf
$\varphi_*W$ on $X$. Note that this action commutes with the trivial action
of $\text{GL}(r,{\mathbb C})$ on $X\,=\,F_{\text{GL}(r,
{\mathbb C})}/\text{GL}(r,{\mathbb C})$.

The vector bundle underlying the parabolic vector bundle corresponding to
$F_{\text{GL}(r, {\mathbb C})}$ is
$$
E_0\,:=\, (\varphi_*W)^{\text{GL}(r,{\mathbb C})}\, \subset\,
\varphi_*W\, .
$$
Here $(\varphi_*W)^{\text{GL}(r,{\mathbb C})}$ denotes the sheaf of
invariants; from the given conditions on $F_{\text{GL}(r, {\mathbb C})}$ it
follows that $E_0$ is a locally free coherent sheaf. We will construct
a parabolic structure on $E_0$.

For any $i\, \in\,[1\, ,\ell]$, the reduced divisor $\varphi^{-1}(D_i)_{\text{red}}
\, \subset\, F_{\text{GL}(r, {\mathbb C})}$ is preserved by the action
of $\text{GL}(r,{\mathbb C})$ on $F_{\text{GL}(r, {\mathbb C})}$. Therefore,
the line bundle 
$$
{\mathcal O}_{F_{\text{GL}(r, {\mathbb C})}} (\varphi^{-1}(D_i)_{\text{red}})
\, \longrightarrow\, F_{\text{GL}(r, {\mathbb C})}
$$
is equipped with a lift of the action of $\text{GL}(r, {\mathbb C})$
on $F_{\text{GL}(r, {\mathbb C})}$. For each $n\,\in\, \mathbb Z$, the action on
$\text{GL}(r, {\mathbb C})$ on ${\mathcal O}_{F_{\text{GL}(r,
{\mathbb C})}} (\varphi^{-1}(D_i)_{\text{red}})$ produces an action of
$\text{GL}(r, {\mathbb C})$ on the line bundle ${\mathcal O}_{F_{\text{GL}(r,
{\mathbb C})}} (n\cdot \varphi^{-1}(D_i)_{\text{red}})$.

For any $i\, \in\,[1\, ,\ell]$, take any point $x_i\, \in\,
D^{\rm sm}\bigcap D_i$, where $D^{\rm sm}$ as before is the smooth
locus of $D$. Recall that the order of the cyclic isotropy subgroup
$G_z\, \in\, \text{GL}(r, {\mathbb C})$, where $z\, \in\, \varphi^{-1}(x_i)$,
is independent of the choices of both $x_i$ and $z$. Let $n_i$ be the
order of $G_z$, where $z$ is as above.

For any real number $\lambda$, by $[\lambda]$ we will denote the integral
part of $\lambda$. So, $[\lambda]\, \in\, \mathbb Z$, and
$0\, \leq\, \lambda-[\lambda]\, <\, 1$.

For any $t\, \in\, \mathbb R$, consider the vector bundle
$$
W_t\, :=\, W\otimes {\mathcal O}_{F_{\text{GL}(r, {\mathbb C})}} (
\sum_{i=1}^\ell [-tn_i]\cdot \varphi^{-1}(D_i)_{\text{red}})
\, \longrightarrow\, F_{\text{GL}(r, {\mathbb C})}\, ,
$$
where $n_i$ is defined above.
The actions of $\text{GL}(r, {\mathbb C})$ on $W$ and
${\mathcal O}_{F_{\text{GL}(r, {\mathbb C})}} (\varphi^{-1}(D_i)_{\text{red}})$
together produce an action of $\text{GL}(r, {\mathbb C})$ on the vector bundle
$W_t$ defined above. This action of $\text{GL}(r, {\mathbb C})$ on
$W_t$ lifts the action of $\text{GL}(r, {\mathbb C})$ on
$F_{\text{GL}(r, {\mathbb C})}$.

Let
$$
E_t\, :=\,  (\varphi_*W_t)^{\text{GL}(r,{\mathbb C})}\, \subset\,
\varphi_*W_t
$$
be the invariant direct image. This $E_t$ is a locally free coherent sheaf
on $X$. 

This filtration of coherent sheaves $\{E_t\}_{t\in \mathbb R}$ defines a parabolic
vector bundle on $X$ with $E_0$ as thr underlying vector bundle (see \cite{MY}
for the description of a parabolic vector bundles as a filtration
of sheaves). The proof
of it similar to the proofs in \cite{Bi1}, \cite{Bo1}, \cite{Bo2}.

The above construction of a parabolic vector bundle of rank
$r$ from a ramified principal ${\text{GL}(r, {\mathbb C})}
$--bundle is the inverse of the earlier construction of a
ramified principal ${\text{GL}(r, {\mathbb C})}$--bundle from
a parabolic vector bundle.

We note that the above construction of a parabolic vector bundle of rank $r$ from
a ramified principal $\text{GL}(r, {\mathbb C})$--bundle coincides with the
following construction (however we do not need this here for our purpose).

As before, let $F_{\text{GL}(r, {\mathbb C})}\, \longrightarrow\,
X$ be a ramified principal $\text{GL}(r, {\mathbb C})$--bundle.
Then there is a finite (ramified) Galois covering
$$
\gamma\, :\, Y\, \longrightarrow\, X
$$
such that the normalization $\widetilde{F_{\text{GL}(r, {\mathbb C})}
\times_X Y}$ of the fiber product
$F_{\text{GL}(r, {\mathbb C})}\times_X Y$ is smooth. The projection 
$\widetilde{F_{\text{GL}(r, {\mathbb C})}\times_X Y}
\,\longrightarrow\, Y$ is a principal $\text{GL}(r,
{\mathbb C})$--bundle equipped with an action of the Galois
group $\Gamma\,:=\,
\text{Gal}(\gamma)$. Let $F_{V_0}\,:=\, \widetilde{F_{\text{GL}(r, 
{\mathbb C})}\times_X Y}(V_0)$ be the vector bundle over
$Y$ associated to the principal $\text{GL}(r,
{\mathbb C})$--bundle $\widetilde{F_{\text{GL}(r, {\mathbb C})}
\times_X Y}$ for the standard $\text{GL}(r, {\mathbb C})$--module
$V_0 \,:=\, {\mathbb C}^r$. The action of $\Gamma$
on $Y$ induces an action of $\Gamma$ on $\widetilde{F_{\text{GL}(r,
{\mathbb C})}\times_X Y}$; this action of $\Gamma$ on
$\widetilde{F_{\text{GL}(r, {\mathbb C})}\times_X Y}$ 
commutes with the action of
${\text{GL}(r, {\mathbb C})}$ on $\widetilde{F_{\text{GL}(r, 
{\mathbb C})}\times_X Y}$. Hence the action of $\Gamma$ on
$\widetilde{F_{\text{GL}(r, {\mathbb C})}\times_X Y}$ induces 
an action of $\Gamma$ on the above defined associated bundle
$F_{V_0}$ making $F_{V_0}$ a $\Gamma$--linearized vector
bundle. Let $E_*$ be the parabolic vector bundle of rank $r$
over $X$ associated
to this $\Gamma$--linearized vector bundle $F_{V_0}$.

Take an irreducible component $D_i$ of the parabolic divisor $D$.
Consider the parabolic vector bundle $E_*$ constructed above from
the ramified principal $\text{GL}(r, {\mathbb C})$--bundle
$F_{\text{GL}(r, {\mathbb C})}\, \longrightarrow\, X$. A rational
number $0\, \geq\, \lambda\, <\, 1$ is a parabolic weight for the
quasiparabolic filtration of $E_*$ over $D_i$ if and only if
$\exp(2\pi\sqrt{-1} \lambda)$ is an eigenvalue of the isotropy subgroup
$G_z$ for a general point $z$ of $D_i$; if $\lambda$ is a parabolic
weight, then its multiplicity coincides with the multiplicity of the
eigenvalue $\exp(2\pi\sqrt{-1} \lambda)$ of $G_z$.

\section{Chern classes of parabolic vector bundles}

\subsection{Projective bundle and the tautological line bundle}

Let $E_*$ be a parabolic vector bundle over $X$ of rank $r$. Let
\begin{equation}\label{ee}
\varphi\, :\, E_{\text{GL}(r, {\mathbb C})}\,\longrightarrow\, X
\end{equation}
be the corresponding ramified principal
$\text{GL}(r, {\mathbb C})$--bundle. Let ${\mathbb P}^{r-1}$ be
the projective space parametrizing the
hyperplanes in ${\mathbb C}^r$. The standard
action of $\text{GL}(r, {\mathbb C})$ on ${\mathbb C}^r$ produces
an action of $\text{GL}(r, {\mathbb C})$ on ${\mathbb P}^{r-1}$.
Let
\begin{equation}\label{rho}
\rho\, :\, \text{GL}(r, {\mathbb C})\, \longrightarrow\,
\text{Aut}({\mathbb P}^{r-1})
\end{equation}
be the homomorphism defined by this action. Let 
\begin{equation}\label{e7}
{\mathbb P}(E_*) \, = \, E_{\text{GL}(r, {\mathbb C})}({\mathbb P}^{r-1})
\, :=\, E_{\text{GL}(r, {\mathbb C})}\times^{\text{GL}(r,
{\mathbb C})} {\mathbb P}^{r-1}\,\longrightarrow\, X
\end{equation}
be the associated (ramified) fiber bundle. We note that
${\mathbb P}(E_*)$ is a quotient of $E_{\text{GL}(r, {\mathbb 
C})}\times{\mathbb P}^{r-1}$; two points $(y_1\, ,z_1)$ and
$(y_2\, ,z_2)$ of $E_{\text{GL}(r, {\mathbb
C})}\times{\mathbb P}^{r-1}$ are identified in ${\mathbb P}(E_*)$
if there is an element $g\, \in\, \text{GL}(r, {\mathbb C})$ such that
$y_2\,=\, y_1g$ and $z_2\,=\, \rho(g^{-1})(z_1)$, where $\rho$ is
the homomorphism in \eqref{rho}.

\begin{definition}\label{def1}
{\rm We will call ${\mathbb P}(E_*)$ the} projective bundle {\rm
associated to the parabolic vector bundle $E_*$.}
\end{definition}

Take a point $x\,\in\,D$; it need not be a smooth point of $D$.
Take any $z\,\in\, \varphi^{-1}(x)$, where $\varphi$
is the morphism in \eqref{ee}. As in \eqref{e8}, let
$G_z\,\subset\, \text{GL}(r, {\mathbb C})$
be the isotropy subgroup of $z$
for the action of $\text{GL}(r, {\mathbb C})$ on $E_{\text{GL}(r, 
{\mathbb C})}$. We recall that $G_z$ is a finite group. Let $n_x$
be the order of $G_z$; we note that $n_x$ is independent of
the choice of $z$ in $\varphi^{-1}(x)$ because for any
other $z'\, \in\, \varphi^{-1}(x)$, the two groups $G_{z'}$
and $G_z$ are isomorphic. The number of
distinct integers $n_x$, $x\,\in\,D$, is finite.
Indeed, this follows immediately from the fact that as $x$ moves over a 
fixed connected component of $D^{\rm sm}$, the conjugacy
class of the subgroup $G_z\, \subset\, \text{GL}(r, {\mathbb C})$,
$z\, \in\, \varphi^{-1}(x)$, remains unchanged.

Let
\begin{equation}\label{e9}
N(E_*)\, :=\, \text{l.c.m.}\{n_x\}_{x\in D}
\end{equation}
be the least common multiple of all these integers $n_x$.

As before, ${\mathbb P}^{r-1}$ is the projective
space parametrizing the hyperplanes in ${\mathbb C}^r$.
For any point $y\,\in\, {\mathbb P}^{r-1}$, let
\begin{equation}\label{e10}
H_y\,\subset\, \text{GL}(r, {\mathbb C})
\end{equation}
be the isotropy subgroup for the action of $\text{GL}(r, {\mathbb C})$
on ${\mathbb P}^{r-1}$ constructed using $\rho$ in \eqref{rho}. So
$H_y$ is a maximal parabolic subgroup of $\text{GL}(r, {\mathbb C})$.
Let ${\mathcal O}_{{\mathbb P}^{r-1}}(1)\,\longrightarrow \,
{\mathbb P}^{r-1}$ be the tautological quotient line bundle. The group
$H_y$ in \eqref{e10} acts on the fiber
${\mathcal O}_{{\mathbb P}^{r-1}}(1)_y$ over the point $y$.

For points $x\, \in\, D$, $z\, \in\, \varphi^{-1}(x)$ and
$y\,\in\, {\mathbb P}^{r-1}$, the group $G_z\bigcap H_y\,\subset\, 
\text{GL}(r, {\mathbb C})$, where $G_z$ and $H_y$ are defined in 
\eqref{e8} and \eqref{e10} respectively, acts trivially on the fiber
${\mathcal O}_{{\mathbb P}^{r-1}}(n_x)_y$ of the line bundle
${\mathcal O}_{{\mathbb P}^{r-1}}(n_x)\,:=\,
{\mathcal O}_{{\mathbb P}^{r-1}}(1)^{\otimes n_x}$
over $y$. Indeed, this
follows from the fact that $n_x$ is the order of $G_z$. 
Therefore, from the definition of $N(E_*)$ in \eqref{e9} it follows
immediately that for any
$z\,\in\, \varphi^{-1}(D)$ and any $y\,\in\, {\mathbb P}^{r-1}$,
the group $G_z\bigcap H_y\,\subset\, \text{GL}(r, {\mathbb C})$
acts trivially on the fiber of the line bundle
$${\mathcal O}_{{\mathbb P}^{r-1}}(N(E_*))\, :=\,
{\mathcal O}_{{\mathbb P}^{r-1}}(1)^{\otimes N(E_*)}
$$
over the point $y$.

Consider the action of $\text{GL}(r, {\mathbb C})$ on the total space
of the line bundle ${\mathcal O}_{{\mathbb P}^{r-1}}
(N(E_*))$ constructed using the standard action of
$\text{GL}(r, {\mathbb C})$ on ${\mathbb C}^r$. Let
$$
E_{\text{GL}(r, {\mathbb C})}({\mathcal O}_{{\mathbb P}^{r-1}}
(N(E_*)))
\, :=\, E_{\text{GL}(r, {\mathbb C})}\times^{\text{GL}(r, {\mathbb C})}
{\mathcal O}_{{\mathbb P}^{r-1}} 
(N(E_*)) \,\longrightarrow\, X
$$
be the associated fiber bundle. Since the natural projection
$$
{\mathcal O}_{{\mathbb P}^{r-1}} (N(E_*))\,\longrightarrow\,
{\mathbb P}^{r-1}
$$
intertwines the actions of $\text{GL}(r, {\mathbb C})$ on
${\mathcal O}_{{\mathbb P}^{r-1}} (N(E_*))$ and ${\mathbb P}^{r-1}$,
this natural projection produces a projection
\begin{equation}\label{e11}
E_{\text{GL}(r, {\mathbb C})}({\mathcal O}_{{\mathbb P}^{r-1}}
(N(E_*))) \,\longrightarrow\, {\mathbb P}(E_*)
\end{equation}
between the associated bundles, where ${\mathbb P}(E_*)$ is
constructed in \eqref{e7}.

Using the above observation that $G_z\bigcap H_y$ acts trivially on
the fiber of ${\mathcal O}_{{\mathbb P}^{r-1}} (N(E_*))$ over $y$ it
follows immediately that the projection in \eqref{e11} makes
$E_{\text{GL}(r, {\mathbb C})}
({\mathcal O}_{{\mathbb P}^{r-1}} (N(E_*)))$
an algebraic line bundle over the variety ${\mathbb P}(E_*)$.

\begin{definition}\label{def2}
{\rm The line bundle $E_{\text{GL}(r, {\mathbb C})}({\mathcal O}_{
{\mathbb P}^{r-1}}(N(E_*)))\,\longrightarrow\, {\mathbb P}(E_*)$ will be
called the} tautological line bundle; {\rm this tautological line bundle
will be denoted by ${\mathcal O}_{{\mathbb P}(E_*)}(1)$.}
\end{definition}

\subsection{Chern class of the tautological line bundle}

For any nonnegative integer $i$, define the
rational Chow group $\text{CH}^i(X)_{\mathbb
Q}\,:=\, \text{CH}^i(X)\bigotimes_{\mathbb Z} \mathbb Q$.

Let $E_*$ be a parabolic vector bundle over $X$ of rank $r$.
The corresponding ramified principal
$\text{GL}(r, {\mathbb C})$--bundle over $X$ will be denoted by
$E_{\text{GL}(r, {\mathbb C})}$. Consider ${\mathbb P}(E_*)$
constructed as in \eqref{e7} from $E_{\text{GL}(r, {\mathbb C})}$. Let
$$
\psi\, :\, {\mathbb P}(E_*)\, \longrightarrow\, X
$$
be the natural projection. Let ${\mathcal O}_{{\mathbb P}(E_*)}(1)$
be the tautological line bundle over ${\mathbb P}(E_*)$ (see
Definition \ref{def2}).

\begin{theorem}\label{thm1}
For each integer $i\, \in\, [0\, ,r]$, there is a unique element
$$
\widetilde{C}_i(E_*)\, \in\, {\rm CH}^i(X)_{\mathbb Q}
$$
such that
\begin{equation}\label{b3}
\sum_{i=0}^r (-1)^i c_1({\mathcal O}_{{\mathbb P}(E_*)}
(1))^{r-i} \psi^*\widetilde{C}_i(E_*)
\, =\, 0
\end{equation}
with $\widetilde{C}_0(E_*)\, =\, 1/N(E_*)^r$, where $N(E_*)$ is
the integer in \eqref{e9}.
\end{theorem}

\begin{proof}
Let $\gamma\, :\, Y\,\longrightarrow\, X$ be the covering in
\eqref{e1} (recall that it depends on $E_*$).
Let $E'\,\longrightarrow\, Y$ be the
$\Gamma$--linearized vector bundle over $Y$ corresponding to
$E_*$, where $\Gamma\,=\, \text{Gal}(\gamma)$ is the Galois 
group of $\gamma$. Let ${\mathbb P}(E')$ be the projective bundle
over $Y$ parametrizing the hyperplanes in the fibers of $E'$. The 
tautological line bundle over ${\mathbb P}(E')$ will be denoted by
${\mathcal O}_{{\mathbb P}(E')}(1)$.

The action of $\Gamma$ on $E'$ produces
an action of $\Gamma$ on ${\mathbb P}(E')$ lifting
the action of $\Gamma$ on $Y$. It can be seen that the
variety ${\mathbb P}(E_*)$ in \eqref{e7} is the quotient
\begin{equation}\label{g1}
\Gamma\backslash {\mathbb P}(E')\,=\,
{\mathbb P}(E_*)\, .
\end{equation}
Indeed, this follows immediately from the fact
that $\Gamma\backslash E'_{\text{GL}(r, {\mathbb C})}\,=\,
E_{\text{GL}(r, {\mathbb C})}$, where
$E_{\text{GL}(r, {\mathbb C})}$ is the ramified principal
$\text{GL}(r, {\mathbb C})$--bundle corresponding to
$E_*$, and $E'_{\text{GL}(r, {\mathbb C})}$ is the
principal $\text{GL}(r, {\mathbb C})$--bundle corresponding to $E'$.

For any point $y\, \in\, Y$, let $\Gamma_y\, \subset\, \Gamma$ be the
isotropy subgroup that fixes $y$ for the action of $\Gamma$ on $Y$.
The action of $\Gamma_y$ on the fiber of ${\mathcal O}_{{\mathbb P}
(E')}(N(E_*))\, :=\, {\mathcal O}_{{\mathbb P}(E')}(1)^{\otimes
N(E_*)}$ is trivial, where $N(E_*)$ is the integer in \eqref{e9}.
Indeed, this follows immediately from the construction of $E_*$
for $E'$. Therefore, the quotient $\Gamma\backslash 
{\mathcal O}_{{\mathbb P}(E')}(N(E_*))$ defines a line bundle over
$\Gamma\backslash{\mathbb P}(E')\,=\, {\mathbb P}(E_*)$. We have a
natural isomorphism of line bundles
\begin{equation}\label{g2}
\Gamma\backslash{\mathcal O}_{{\mathbb P}(E')}(N(E_*))\,=\,
{\mathcal O}_{{\mathbb P}(E_*)}(1)\, .
\end{equation}

Let
$$
\psi_{E'}\, :\, {\mathbb P}(E')\, \longrightarrow\, Y
$$
be the natural projection. For any $i\, \in\, [0\, ,r]$, let
$$
c_i(E')\,\in\,\text{CH}^i(Y)_{\mathbb Q}\,:=\,\text{CH}^i(Y)
\otimes_{\mathbb Z}{\mathbb Q}
$$
be the $i$--th Chern class of $E'$. We have
\begin{equation}\label{g3}
\sum_{i=0}^r \frac{(-1)^i}{N(E_*)^{r-i}} c_1({\mathcal O}_{{\mathbb 
P}(E')}(N(E_*)))^{r-i} \psi^*_{E'}c_i(E')
\,=\,\sum_{i=0}^r (-1)^i c_1({\mathcal O}_{{\mathbb P}(E')}
(1))^{r-i} \psi^*_{E'}c_i(E')\,=\, 0
\end{equation}
(see \cite[page 429]{Ha}). The identity in \eqref{g3} in
fact uniquely determines the Chern classes of $E'$
provided it is given that $c_0(E') \,=\,1$.

Since the vector bundle $E'$ is $\Gamma$--linearized, it
follows immediately that
\begin{equation}\label{g4}
c_i(E')\,\in\, (\text{CH}^i(Y)_{\mathbb Q})^\Gamma\, ,
\end{equation}
where $(\text{CH}^i(Y)_{\mathbb Q})^\Gamma$ is the invariant
part of $\text{CH}^i(Y)_{\mathbb Q}$ for the action of $\Gamma$
on it. We also know that the pullback homomorphism
$$
\gamma^*\, :\, \text{CH}^i(X)_{\mathbb Q}\,
\longrightarrow\,(\text{CH}^i(Y)_{\mathbb Q})^\Gamma
$$
is an isomorphism \cite[pages 20--21, Example 1.7.6]{Fu}.

{}From \eqref{g1} we have the quotient map
$$
\beta\, :\,{\mathbb P}(E')\,\longrightarrow \,
{\mathbb P}(E_*)
$$
for the action of $\Gamma$, and from \eqref{g2} it follows that
\begin{equation}\label{b1}
\beta^*{\mathcal O}_{{\mathbb P}(E_*)}(1)\,=\,
{\mathcal O}_{{\mathbb P}(E')}(N(E_*))\, .
\end{equation}
Hence
$$
\beta^*c_1({\mathcal O}_{{\mathbb P}(E_*)}(1))\,=\,
c_1({\mathcal O}_{{\mathbb P}(E')}(N(E_*)))\, .
$$
Therefore, from \eqref{g3} and \eqref{g4} we conclude
that for each $i\, \in\, [0\, ,r]$, there is an unique element
$\widetilde{C}_i\, \in\, \text{CH}^i(X)_{\mathbb Q}$
such that $\widetilde{C}_0(E_*)\, =\, 1/N(E_*)^r$ and
$$
\sum_{i=0}^r (-1)^i c_1({\mathcal O}_{{\mathbb P}(E_*)}
(1))^{r-i} \psi^*\widetilde{C}_i(E_*)
\, =\, 0\, .
$$
This completes the proof of the theorem.
\end{proof}

\begin{definition}\label{def4}
{\rm For any integer $i\, \in\, [0\, ,r]$, the}
$i$--th Chern class $c_i(E_*)$ {\rm of a parabolic vector
bundle $E_*$ is defined to be
$$
c_i(E_*)\, :=\, N(E_*)^{r-i}\cdot \widetilde{C}_i(E_*)\, \in\, 
\text{CH}^i(X)_{\mathbb Q}\, ,
$$
where $\widetilde{C}_i(E_*)$ is the class in Theorem \ref{thm1}.}
\end{definition}

\begin{corollary}\label{cor1}
Let $E_*$ be a parabolic vector bundle over $X$ of rank $r$.
Let $E'\,\longrightarrow\, Y$ be the corresponding
$\Gamma$--linearized vector bundle (see the proof of Theorem
\ref{thm1}). Then
$$
c_i(E')\, =\, \gamma^*c_i(E_*)
$$
for all $i$.
\end{corollary}

\begin{proof}
{}From the construction of \eqref{b3} using \eqref{g3} it follows
immediately that $\gamma^* \widetilde{C}_i(E_*) \,=\, 
c_i(E')/N(E_*)^{r-i}$.
Therefore, the corollary follows from Definition \ref{def4}.
\end{proof}

Define the \textit{Chern polynomial} for $E_*$ to be
$$
c_t(E_*)\,=\, \sum_{i=0}^r c_i(E_*)t^i\, ,
$$
where $r\, =\, \text{rank}(E_*)$, and $t$ is a formal variable.
The \textit{Chern character} of $E_*$ is constructed from the
Chern classes of $E_*$ in the following way:
if $c_t(E_*)\,=\, \prod_{i=1}^r (1+\alpha_i t)$, then
$$
{\rm ch}(E_*)\,:=\, \sum_{j=1}^r\exp(\alpha_j)
\,\in\, \text{CH}^*(X)_{\mathbb Q}\, .
$$

\begin{proposition}\label{prop1}
Let $E_*$ and $F_*$ be parabolic vector bundles on $X$.
\begin{enumerate}
\item The Chern polynomial of the parabolic direct sum
$E_*\oplus F_*$ satisfies
the identity $c_t(E_*\oplus F_*)\,=\, c_{t}(E_*)\cdot c_{t}(F_*)$.

\item The Chern polynomial of the parabolic dual $E^*_*$ satisfies
the identity $c_t(E^*_*)\,=\, c_{-t}(E_*)$.

\item{} The Chern character of the parabolic tensor product
$E_*\otimes F_*$ satisfies the identity
${\rm ch}(E_*\otimes F_*)\,=\,{\rm ch}(E_*)\cdot{\rm ch}(E_*)$.
\end{enumerate}
\end{proposition}

\begin{proof}
The Chern classes of usual vector bundles satisfy the above
relations. The correspondence between the parabolic
vector bundles and the $\Gamma$--linearized vector bundles
takes the tensor product (respectively, direct sum) of any two 
$\Gamma$--linearized vector
bundles to the parabolic tensor product (respectively, parabolic
direct sum) of the corresponding parabolic vector bundles. Similarly,
the dual of a given parabolic vector bundle corresponds to the dual
of the $\Gamma$--linearized vector bundle corresponding to the
given parabolic vector bundle. In view of these facts, the
proposition follows form Corollary \ref{cor1}.
\end{proof}

\section{Comparison with equivariant Chern classes}

Let us recall the basic construction of equivariant intersection theory
as in \cite{EG}. Consider a smooth variety $Z$ equipped with an action 
of a finite group $G$. Let $V$ be a representation of $G$ such that 
there is an open subset $U$ of $V$ on which $G$
acts freely and the codimension of the complement $V-U$ is at
least $\dim Z - i$. Following Edidin and Graham we write 
$$
Z_G \,= \,(Z \times U )/G\, .
$$
The equivariant Chow groups are defined to be
$$
A^{i}_G(Z) \,=\, A^{i}(Z_G)\otimes {\mathbb Q}\, .
$$
It is shown in Proposition 1 of \cite{EG} that this 
definition does not depend on $V$ and $U$.

Consider a parabolic vector bundle $E_*$ on $X$.
Let $\gamma\, :\, Y\longrightarrow X$ be a Galois cover as in
the proof of Theorem \ref{thm1}. The Galois group of $\gamma$
will be denoted by $G$. Let $E'$ be the $G$--linearized vector
bundle on $Y$ associated to $E_*$.
The vector bundle $E'$ has equivariant Chern classes
$$
c_i(E') \,\in \, A^i_G(Y)\, .
$$

We have a diagram
$$
\xymatrix{
{\mathbb P}(E_*) \ar[d]^{\psi_X} & {\mathbb P}(E_*)_G = (\Pr(E_*) 
\times U)/G \ar[l]_(.7){\pi_X} \ar[d]^{\psi^G_X} 
& {\mathbb P}(E')_G \ar[l]_(.28){\beta^G} \ar[d]^{\psi_Y^G} \\
X & X_G = (X \times U)/G \ar[l]^(.7){p_X}& Y_G 
\ar[l]^(.3){f_G}
}
$$
Note that the morphisms $p_X$ and $\pi_X$ are flat. Further, the 
morphisms $\beta_G$ and $f_G$ are flat and proper. The scheme 
${\mathbb P}(E')_G$ is a projective bundle over $Y_G$ as the action of 
$G$ on 
$Y\times U$ is free; see also \cite[Lemma 1]{EG}. All these can be 
deduced by using the fact that the group $G$ acts freely on
$X\times U$ and $Y\times U$.

\begin{proposition}
We have the following relationship amongst Chern classes:
$$
p_X^*(c_i(E_*)) \,= \,f_{G,*}(c^G_i(E'))\, .
$$
\end{proposition}

\begin{proof}
By the projection formula it suffices to show that
\begin{equation}\label{l1}
f_G^* p_X^*(c_i(E_*)) \,=\, c^G_i(E')\, .
\end{equation}
Flat pullback preserves intersection products so the equation
obtained by pulling back \eqref{b3} to
${\mathbb P}(E')$ remains valid. As was observed
in the proof of Theorem \ref{thm1} we have that
$$
\beta^{G*}\pi_X^*(\mathcal{O}_{{\mathbb P}(E_*)}(1)) \,=\, 
\mathcal{O}_{{\mathbb P}(E')}(N(E_*))
$$
(see \eqref{b1}). Now using Definition \ref{def4} it is
deduced that \eqref{l1} holds.
\end{proof}

%%%%%%%%%%%%%%%%%%%%%%%%%%%%%%%%%%%%%%%%%%%%%%%%%%%%%%%%%%%%%

\end{document}